\documentclass[11pt]{amsart}
\usepackage{amssymb}
\usepackage{amsmath}
\usepackage{amsfonts}
\usepackage{amscd}

\newtheorem{theorem}{Theorem}[section]
\newtheorem{definition}[theorem]{Definition}
\newtheorem{lemma}[theorem]{Lemma}
\newtheorem{remark}[theorem]{Remark}

\theoremstyle{definition}
\theoremstyle{remark}
\numberwithin{equation}{section}

\begin{document}
\title[Some Subspaces of an FK Space and Deferred Ces\`{a}ro Conullity]{Some Subspaces of an FK Space and Deferred Ces\`{a}ro Conullity }
\author{\.{I}. Da\v{g}adur}
\address{Mersin University, Faculty of Science and Literature, Department of
Mathematics, 33343 Mersin - TURKEY.}
\email{seydasezgek@gmail.com}
\email{ilhandagadur@yahoo.com}
\author{\c{S}. SEZGEK}
\maketitle

\begin{abstract}
In this paper, we construct new important the subspaces $D_{p}^{q}S$, $D_{p}^{q}W$, $D_{p}^{q}F$ and $D_{p}^{q}B$ for a locally convex FK-space $X$ containing $\phi$, the space of finite sequences. Then we show that there is relation among these subspaces. Also, we study deferred Ces\`{a}ro conullity of one FK-space with respect to another, and we give some important results. Finally, we examine the deferred Ces\`{a}ro conullity of the absolute summability domain $l_A$, and show that if $l_A$ is deferred Ces\`{a}ro conull, then $A$ cannot be $l$-replaceable.
\end{abstract}

~~  

\newline


\section{introduction}

Let $w$ denote the space of all complex valued sequences. It can be topologized with the seminorms $r_n(x)=|x_n|$, $n=1,2,\ldots$, and any vector subspace $X$ of $w$ is a sequence space. A sequence space $X$ with a vector space topology $\tau$ is a K-space provided that the inclusion map $i:(X,\tau)\to w$, $i(x)=x$, is continuous. If, in addition, $\tau$ is complete, matrizable, locally convex then $(X,\tau)$ is called FK-space. So an FK-space is a complete, metrizable locally convex topological vector space of sequences for which the coordinate functionals are continuous. An FK-space whose topology is normable is called a BK-space. The basic properties of FK-space may be found in (see \cite{boos}, \cite{wilansky} and \cite{zeller}).  

By $c$, $c_0$, $l_{\infty}$ we denote the spaces of convergent sequences, null sequences and bounded sequences, respectively. These are FK-spaces under $\|x\|=\sup_n|x_n|$.  By $cs$, $l$ we denote the spaces of all summable sequences, absolute summable sequences, respectively. 

Throughout this paper $e$ denotes the sequences of ones; $\delta^j$ $(j=1,2,\ldots)$ the sequence with the one in the $j$-th position; $\phi$ the linear span of $\delta^j$'s.  The linear span of $\phi$ and $e$ is denoted by $\phi_1$. The topological dual of $X$ is denoted by $X'$. The space $X$ is said to have AD if $\phi$ is dense in $X$. A sequence $x$ in a locally convex sequence space $X$ is said the property AK if $x^{(n)}\to x$ in $X$ where $x^{(n)}=\sum_{k=1}^{n}x_k\delta^k$.

We recall (see \cite{boos} and \cite{wilansky}) that the $f$, $\beta$-duals of a subset $X$ of $w$ are
\begin{eqnarray}
&&X^f=\left\{ \{f(\delta^k)\} : f\in X' \right\}, \nonumber \\
X^{\beta}&=&\left\{x\in w : \sum_{n=1}^{\infty}x_ky_k \text{ is convergent for all } y\in X \right\}  . \nonumber \end{eqnarray}

In 1932, Agnew \cite{agnew} defined the deferred Ces\`{a}ro mean $D_{p,q}$ of the sequences $x$ by
\begin{equation*}
(D_{p,q}x)_n=\frac{1}{q(n)-p(n)}\sum_{k=p(n)+1}^{q(n)}x_k \  \  \ 
\end{equation*}
where $\{p(n)\}$ and $\{q(n)\}$ are sequences of nonnegative integers satisfying the conditions $p(n)<q(n)$ and $\underset{n\to \infty}{\lim}q(n)=\infty$. We note here that $D_{p,q}$ is clearly regular for any choice of $\{p(n)\}$ and $\{q(n)\}$. The deferred Ces\`{a}ro mean is used throughout this paper. We define some new sequence space by using deferred Ces\`{a}ro mean.

The sequence spaces
\begin{equation*}
[\sigma_0]_{p}^{q}:= \left\{ x\in w : \lim_n \frac{1}{q(n)-p(n)} \sum_{k=p(n)+1}^{q(n)}x_k  =0 \right\},
\end{equation*}
\begin{equation*}
[\sigma_c]_{p}^{q}:= \left\{ x\in w : \lim_n \frac{1}{q(n)-p(n)} \sum_{k=p(n)+1}^{q(n)}x_k ~~ \text{  exists} \right\},
\end{equation*}
\begin{equation*}
[\sigma_{\infty}]_{p}^{q}:= \left\{ x\in w : \sup_n  \frac{1}{q(n)-p(n)}\left| \sum_{k=p(n)+1}^{q(n)}x_k \right| <\infty \right\},
\end{equation*}

\begin{equation*}
\sigma_{p}^{q} [s]:= \left\{ x\in w : \lim_n \frac{1}{q(n)-p(n)} \sum_{k=p(n)+1}^{q(n)}\sum_{j=1}^{k}x_j  ~~ \text{  exists} \right\}
\end{equation*}
and
\begin{equation*}
\sigma_{p}^{q} [b]:= \left\{ x\in w : \sup_n \left|  \frac{1}{q(n)-p(n)} \sum_{k=p(n)+1}^{q(n)}\sum_{j=1}^{k}x_j\right|  <\infty \right\}
\end{equation*}
are BK-spaces with the norms
\begin{equation*}
\parallel x\parallel_{[\sigma_0]_{p}^{q}}=\sup_n\left| \frac{1}{q(n)-p(n)} \sum_{k=p(n)+1}^{q(n)}x_k  \right| \  
\end{equation*}
and
\begin{equation*}
\parallel x\parallel_{\sigma_{p}^{q} [s]}=\sup_n\left| \frac{1}{q(n)-p(n)} \sum_{k=p(n)+1}^{q(n)}\sum_{j=1}^{k}x_j  \right| \  .
\end{equation*}	
The proof follows the same lines as in (see \cite{buntinas}, \cite{dagadur}, \cite{goes} and \cite{goes2}), so we omit the
details.

A sequence $x$ in a locally convex sequence space $X$ is said the property  $\sigma_{p}^{q} [K]  $   if 
\begin{equation*}
\frac{1}{q(n)-p(n)}\sum_{k=p(n)+1}^{q(n)}x^{(k)}\to x \text{  in  } X \ .
\end{equation*}  

A sequence $x$ in a locally convex sequence space $X$ is said the property  $\sigma_{p}^{q} [B]  $   if $\forall x\in X$
\begin{equation*}
\left\{ \frac{1}{q(n)-p(n)}\sum_{k=p(n)+1}^{q(n)}x^{(k)} \right\}   \text{  is bounded in  } X \ .
\end{equation*}

Now we determine a new $d$-, $d[b]$-type duality of a sequence space $X$ containing $\phi$.
\begin{eqnarray}
X^{d} &=& \left\{ x\in w  :  \lim_{n\to\infty}\frac{1}{q(n)-p(n)}\sum_{k=p(n)+1}^{q(n)}\sum_{j=1}^{k}x_jy_j \text{ exists for all } y\in X  \right\}  \nonumber \\
&=&\left\{ x\in w : x.y\in \sigma_{p}^{q} [s] \text{ for all } y\in X  \right\} ,\nonumber
\end{eqnarray} 
\begin{eqnarray}
X^{d[b]} &=& \left\{ x\in w  :  \sup_{n}\frac{1}{q(n)-p(n)} \left| \sum_{k=p(n)+1}^{q(n)}\sum_{j=1}^{k}x_jy_j\right| <\infty \ ,~~  y\in X  \right\} \nonumber \\
&=&\left\{ x\in w : x.y\in \sigma_{p}^{q} [b] \text{ for all }   y\in X  \right\} , \nonumber
\end{eqnarray} 
respectively, where $x.y=(x_ny_n)$.

Let $X$, $Y$ be sets of sequences. Then for $\nu= f, \beta, b, d[b]$ 

i) $X\subset X^{\nu\nu}$,

ii) $X^{\nu\nu\nu}=X^{\nu}$,

iii) If $X\subset Y$ then $Y^{\nu}\subset X^{\nu}$ holds.
\begin{theorem}
	Let $X$ be an FK-space containing $\phi$ and $\lim_{n\to\infty}\frac{q(n)-i+1}{q(n)-p(n)}=1$   $(i\leq q(n))$. Then
	
	i) $X^{\beta}\subset X^{d}\subset X^{d[b]}\subset X^f$;
	
	ii) if $X$ is $\sigma_{p}^{q} [K]$-space then $X^f=X^{d}$;
	
	iii) if $X$ is AD-space then $X^{d[b]}=X^{d}$.
\end{theorem}

\begin{proof}
	ii) 	Let $u\in X^{d}$ and $f(x)=\lim_{n\to\infty}\frac{1}{q(n)-p(n)}\sum_{k=p(n)+1}^{q(n)}\sum_{j=1}^{k}x_ju_j$ for $x\in X$. Then $f\in X'$ by Banach-Steinhaus Theorem [\cite{wilansky}; Theorem 1.0.4]. Now we get
	\begin{equation*}
	f(\delta^i)=\lim_{n\to\infty}\frac{1}{q(n)-p(n)}\sum_{k=p(n)+1}^{q(n)}\sum_{j=1}^{k}u_j\delta^i=\lim_{n\to\infty}\frac{q_n-i+1}{q(n)-p(n)}u_i=u_i
	\end{equation*} 
	so $u\in X^f$. Thus $X^{d}\subset X^f $. 
	
	Let $u\in X^f$. Since $X$ is $\sigma_{p}^{q} [K]$-space 
	\begin{equation*}
	f(x)=\lim_{n\to\infty}\frac{1}{q(n)-p(n)}\sum_{k=p(n)+1}^{q(n)}\sum_{j=1}^{k}x_jf(\delta^j)=\lim_{n\to\infty}\frac{1}{q(n)-p(n)}\sum_{k=p(n)+1}^{q(n)}\sum_{j=1}^{k}x_ju_j
	\end{equation*} 	
	for $x\in X$, then $u\in X^{d} $. Hence $X^f=X^{d}$. 
	
	iii) 	Let $u\in X^{d[b]}$. We define 
	\begin{equation*}
	f_n(x)=\frac{1}{q(n)-p(n)}\sum_{k=p(n)+1}^{q(n)}\sum_{j=1}^{k}u_jx_j
	\end{equation*}
	for $x\in X$. Then $\{f_n\}$ is pointwise bounded and hence equicontinuous by Theorem 7.0.2 of \cite{wilansky}. Since 
	\begin{equation*}
	\lim_{n\to\infty} f_n(\delta^{i})=u_i \ , \ ~~ i<q(n),
	\end{equation*}
	we conclude that $\phi\subset \{x: \lim_nf_n(x) \text{ exists } \}$ is a closed subspace of $X$ by the Convergence Lemma (see \cite{wilansky}; 1.0.5 and 7.0.3). Since $X$ is an AD-space, $X=\{x: \lim_nf_n(x)  \text{ exists } \}=\overline{\phi}$ and thus $\lim_nf_n(x)$ exists for all $x\in X$. Therefore, $u\in X^{d}$. The opposite inclusion is trivial.
	
	i)  $\overline{\phi}\subset X$ by the hypothesis. Since $\overline{\phi}$ is an AD-space, we find 
	\begin{equation*}
	X^{d[b]}\subset (\overline{\phi})^{d[b]}=(\overline{\phi})^{d}\subset (\overline{\phi})^f=X^f
	\end{equation*}
	by (ii), (iii) and Theorem 7.2.4 of \cite{wilansky}.
\end{proof}

\section{Main Results}

We shall define some new subspaces of a locally convex FK-space $X$ containing $\phi$, the space of finite sequences, which are the impotance of each one on topological sequence spaces theory.

\begin{definition}
	\label{subseq}
	Let $X$ be an FK-space $\supset \phi$. Then
	\begin{eqnarray}
	D_{p}^{q}W&:=&D_{p}^{q}W(X)= \left\{ x\in X  :  \frac{1}{q(n)-p(n)}\sum_{k=p(n)+1}^{q(n)}x^{(k)} \to x~ \text{(weakly)} \text{  in  } X     \right\}  \nonumber \\
	&= &\left\{ x\in X  :  f(x)=\lim_n \frac{1}{q(n)-p(n)}\sum_{k=p(n)+1}^{q(n)}\sum_{j=1}^{k}x_jf(\delta^j) ~ \text{for all  } f\in  X     \right\} ~~ , \nonumber
	\end{eqnarray}
	
	\begin{eqnarray}
	D_{p}^{q}S:=D_{p}^{q}S(X)&=& \left\{ x\in X  :  \frac{1}{q(n)-p(n)}\sum_{k=p(n)+1}^{q(n)}x^{(k)} \to x~    \right\}  \nonumber \\
	&=& \left\{ x\in X  :  x \text{  has  }  \sigma_{p}^{q} [K]  \text{  in  } X   \right\}  \nonumber \\
	&=&  \left\{ x\in X  : x=\lim_n \frac{1}{q(n)-p(n)}\sum_{k=p(n)+1}^{q(n)}\sum_{j=1}^{k} x_j\delta^j~    \right\}.  \nonumber
	\end{eqnarray}
	
	Thus $X$ is an $\sigma_{p}^{q} [K]$-space if and only if $D_{p}^{q}S=X$.	
	
	\begin{eqnarray}
	&&	D_{p}^{q}F^+:=D_{p}^{q}F^+(X) \nonumber \\
	&=&\left\{ x\in w  :  \lim_n \left\{    \frac{1}{q(n)-p(n)}\sum_{k=p(n)+1}^{q(n)}x^{(k)}\right\}  \text{ is weakly Cauchy in } X  \right\}  \nonumber \\
	&=& \left\{ x\in w  :       \lim_n   \frac{1}{q(n)-p(n)}\sum_{k=p(n)+1}^{q(n)} \sum_{j=1}^{k} x_j f(\delta^j)   \text{ exists for all } f\in X'    \right\}  \nonumber \\
	&=&  \left\{ x\in w  :  \left\{ x_nf(\delta^n)   \right\} \in \sigma_{p}^{q} [s]  \text{  for all } f\in X'   \right\} = (X^{f})^{d}  \ .     \nonumber
	\end{eqnarray}
	
		\begin{eqnarray}
	D_{p}^{q}B^+:=D_{p}^{q}B^+(X)&=& \left\{ x\in w  : \left\{ \frac{1}{q(n)-p(n)}\sum_{k=p(n)+1}^{q(n)}x^{(k)} \right\} \text{  is bounded in } X    \right\}  \nonumber \\
	&=& \left\{ x\in w  : (x_nf(\delta^n))\in \sigma_{p}^{q} [b]   \text{  for all  } f\in X'   \right\}  \nonumber 
	\end{eqnarray}
	also	$D_{p}^{q}F=D_{p}^{q}F^+\cap X$  and  $	D_{p}^{q}B=	D_{p}^{q}B^+\cap X$.
\end{definition}
 
 We now study some inclusions which are analogous to those given in  ( \cite{wilansky}; chapter 10 ). Also, we prove some theorems related to the $f$-, $d$- and $d[b]$-duality of a sequence space $X$.

\begin{theorem}
	\label{inclusion}
	Let $X$ be an FK-space $\supset \phi$. Then 
	\begin{equation*}
	\phi\subset D_{p}^{q}S \subset D_{p}^{q}W \subset D_{p}^{q}F \subset D_{p}^{q}B\subset X ~~ \text{ and } ~~ \phi\subset D_{p}^{q}S \subset D_{p}^{q}W \subset \overline{\phi}.
	\end{equation*}
\end{theorem}
\begin{proof}
The only non-trivial part is $ D_{p}^{q}W \subset \overline{\phi}$. Let $f\in X'$ and $f=0$ on $\phi$. The definition of $D_{p}^{q}W$ shows that $f=0$ on $D_{p}^{q}W$. Hence, the Hahn-Banach theorem gives the result.	 
\end{proof}	
\begin{theorem}
	\label{monotone}
	The subspaces $E=D_{p}^{q}S,\ D_{p}^{q}W,\ D_{p}^{q}F,\ D_{p}^{q}F^+,\ D_{p}^{q}B$ and $D_{p}^{q}B^+$ of $X$ FK-spaces are monotone i.e., if $X\subset Y$ then $E(X)\subset E(Y)$.
\end{theorem}
\begin{proof}
The inclusion map $i : X \to Y$ is continuous by Corollary 4.2.4 of \cite{wilansky}, so $\frac{1}{q(n)-p(n)}\sum_{k=p(n)+1}^{q(n)}x^{(k)}\to x$	in $X$ implies the same in $Y$. This proves the assertion for $D_{p}^{q}S$. For $D_{p}^{q}W$ it follows from the fact that $i$ is weakly continuous by (4.0.11) of \cite{wilansky}. Now $z \in D_{p}^{q}F^+$, $ D_{p}^{q}B^+ $ if and only if $(z_nf(\delta^n))\in \sigma_{p}^{q} [s]$, $ \sigma_{p}^{q} [b]$ respectively for all $f\in X'$, hence for all $g\in Y'$ since $g|X\in X'$ by Corollary 4.2.4 of \cite{wilansky}. The result follows for $D_{p}^{q}F^+$, $D_{p}^{q}B^+$ and so for $D_{p}^{q}F$, $D_{p}^{q}B$.
\end{proof}

Since $[\sigma_0]_{p}^{q}$ is an AK-space, we immediately get the following
\begin{theorem}
	Let $X$ be an FK-space $\supset [\sigma_0]_{p}^{q} $. Then $[\sigma_0]_{p}^{q} \subset D_{p}^{q}S\subset D_{p}^{q}W $.
\end{theorem}
\begin{theorem}
	\label{DB}
Let $X$ be an FK-space $\supset \phi $. Then $D_{p}^{q}B^+=X^{fd[b]}$.
\end{theorem}
\begin{proof}
By Definition \ref{subseq}, $z\in D_{p}^{q}B^+ $ if and only if $z.u\in \sigma_{p}^{q} [b] $ for each $u \in X^f$. This is precisely the assertion.
\end{proof}

\begin{theorem}
	\label{DBphi}
	Let $X$ be an FK-space $\supset \phi $. Then $D_{p}^{q}B^+ $ is the same for all FK-spaces $Y$ between $\overline{\phi}$ and $X$; i.e., $\overline{\phi}\subset Y\subset X$ implies $D_{p}^{q}B^+(Y)=D_{p}^{q}B^+(X) $. Here the closure of $\phi$ is calculated in $X$.
\end{theorem}
\begin{proof}
By Theorem \ref{monotone} we have $D_{p}^{q}B^+(\overline{\phi})\subset D_{p}^{q}B^+(Y) \subset D_{p}^{q}B^+(X) $. By Theorem \ref{DB} and by (7.2.4) of \cite{wilansky} the first and the last are equal.
\end{proof}

\begin{theorem}
	\label{swphi}
	Let $X$ be an FK-space such that $D_{p}^{q}B \supset \overline{\phi}$. Then $\overline{\phi}$ has $\sigma_{p}^{q} [K]$ 	and $D_{p}^{q}S=D_{p}^{q}W= \overline{\phi}$.
\end{theorem}
\begin{proof}
Suppose first that $X$ has $\sigma_{p}^{q} [B]$. Define $f_n:X\to X$ by
\begin{equation*}
f_n(x)=x-\frac{1}{q(n)-p(n)}\sum_{k=p(n)+1}^{q(n)}x^{(k)}.
\end{equation*}
Then $\{f_n\}$ is pointwise bounded, hence equicontinuous by (7.0.2) of \cite{wilansky}. Since  $f_n\to 0$ on $\phi$ then also $f_n\to 0$ on $\overline{\phi}$ by (7.0.3) of \cite{wilansky}. This is the desired conclusion.
\end{proof}

\begin{theorem}
	\label{Ff}
	Let $X$ be an FK-space $\supset \phi $. Then $D_{p}^{q}F^+=X^{fd}$.
\end{theorem}
\begin{proof}
This may be proved as in Theorem \ref{DB}, with $d$ instead of $d[b]$ .
\end{proof}

\begin{theorem}
	Let $X$ be an FK-space $\supset \phi $. Then $D_{p}^{q}F^+$ is the same for all FK-spaces $Y$ between $\overline{\phi}$ and $X$; i.e., $\overline{\phi}\subset Y\subset X$ implies $D_{p}^{q}F^+(Y)=D_{p}^{q}F^+(X)$ $($The closure
	of $\phi$ is calculated in $X)$.
\end{theorem}
The proof is similar to that of Theorem \ref{DBphi}.

\begin{lemma}
	\label{fphi}
	Let $X$ be an FK-space in which $\overline{\phi}$ has $\sigma_{p}^{q} [K]$. Then $D_{p}^{q}F^+=(\overline{\phi})^{dd}$.
\end{lemma}
\begin{proof}
Observe that $D_{p}^{q}F^+=X^{fd}$ by Theorem \ref{Ff}. Since $X^f=(\overline{\phi})^f$ by Theorem (7.2.4) of \cite{wilansky}, we have $X^{fd}=(\overline{\phi})^{fd}$. Hence, by Theorem 1.9 of \cite{goes3} the result follows.
\end{proof}

An FK-space $X$ is said to have $F\sigma_{p}^{q} [K]$ (functional $\sigma_{p}^{q} [K]$) if $X\subset D_{p}^{q}F^+ $ i.e., $X=D_{p}^{q}F$.

\begin{theorem}
	\label{fk}
	Let $X$ be an FK-space $\supset \phi $. Then $X$ has $F\sigma_{p}^{q} [K]$ if and only if $\overline{\phi}$ has $\sigma_{p}^{q} [K]$ and $X\subset(\overline{\phi})^{dd} $.
\end{theorem}

\begin{proof}
	Necessity. $X$ has $\sigma_{p}^{q} [B]$ since $D_{p}^{q}F\subset D_{p}^{q}B$ so $\overline{\phi}$ has $\sigma_{p}^{q} [K]$ by Theorem \ref{swphi}. The
	remainder of the proof follows from Lemma \ref{fphi}. Sufficiency is given by Lemma \ref{fphi}.
\end{proof}

\begin{theorem}
	Let $X$ be an FK-space $\supset \phi $. The following are equivalent:
	
	i) $X$ has $F\sigma_{p}^{q} [K]$ ,
	
	ii) $X\subset (D_{p}^{q}S)^{dd}$ ,
	
	iii) $X\subset (D_{p}^{q}W)^{dd}$ ,
	
	iv) $X\subset (D_{p}^{q}F)^{dd}$ ,
	
	v) $X^{d}= (D_{p}^{q}S)^{d}$ ,
	
	vi) $X^{d}= (D_{p}^{q}F)^{d}$ .
\end{theorem}

\begin{proof}
Observe that (ii) implies (iii) and (iii) implies (iv) and that they are trivial
since
\begin{equation*}
D_{p}^{q}S\subset D_{p}^{q}W\subset D_{p}^{q}F \ .
\end{equation*}
If (iv) is true, then $X^f\subset (D_{p}^{q}F)^{d}= (X^{f})^{dd}\subset X^{d}$ so (i) is true by Theorem 1.9 of
\cite{goes3}. If (i) holds, then Theorem \ref{fk} implies that $\overline{\phi}=D_{p}^{q}S $ and that (ii) holds. The
equivalence of (v), (vi) with the others is clear.
\end{proof}

\begin{theorem}
	Let $X$ be an FK-space $\supset \phi $. The following are equivalent:
	
	i) $X$ has $S\sigma_{p}^{q} [K]$ ,
	
	ii) $X$ has $\sigma_{p}^{q} [K]$ ,
	
	iii) $X^{d}= X'$ .
\end{theorem}
\begin{proof}
Clearly (ii) implies (i). Conversely if $X$ has $S\sigma_{p}^{q} [K]$ it must have AD for $D_{p}^{q}W\subset \overline{\phi}$ by Theorem \ref{inclusion}. It also has $\sigma_{p}^{q} [B]$ since $D_{p}^{q}W\subset D_{p}^{q}B$. Thus $X$ has $\sigma_{p}^{q} [K]$ by Theorem \ref{swphi}, this proves that (i) and (ii) are equivalent. Assume that (iii) holds. Let $f\in X'$, then there exists $u\in X^{d}$ such that
\begin{equation*}
f(x)= \lim_{n\to\infty} \frac{1}{q(n)-p(n)}\sum_{k=p(n)+1}^{q(n)}\sum_{j=1}^{k}u_jx_j
\end{equation*}
for $x\in X$. Since $f(\delta^j)=u_j$, it follows that each $x\in D_{p}^{q}W$ which shows that (iii) implies (i). That (ii) implies (iii) is known (see \cite{goes} , page 97).
\end{proof}

\begin{theorem}
	Let $X$ be an FK-space $\supset \phi $. The following are equivalent:
	
	i) $D_{p}^{q}W$ is closed in $X$,
	
	ii) $\overline{\phi}\subset D_{p}^{q}B $,
	
	iii) $\overline{\phi}\subset D_{p}^{q}F $,
	
	iv) $\overline{\phi}= D_{p}^{q}W $,
	
	v) $\overline{\phi}= D_{p}^{q}S $,
	
	vi) $D_{p}^{q}S$ is closed in $X$.
\end{theorem}
\begin{proof}
(ii) implies (v): By Theorem \ref{swphi}, $\overline{\phi}$ has $\sigma_{p}^{q} [K]$ , i.e. $\overline{\phi}\subset D_{p}^{q}S$. The opposite inclusion is Theorem \ref{inclusion}. Note that (v) implies (iv), (iv) implies (iii) and (iii) implies
(ii) because
\begin{equation*}
D_{p}^{q}S \subset D_{p}^{q}W \subset \overline{\phi} \ ,~~ D_{p}^{q}W \subset D_{p}^{q}F \subset D_{p}^{q}B \ ;
\end{equation*}
(i) implies (iv) and (vi) implies (v) since $\phi \subset D_{p}^{q}S \subset D_{p}^{q}W \subset \overline{\phi}$. Finally (iv) implies (i) and (v) implies (vi).
\end{proof}

\section{Combinations of Some Subspaces of an FK-Space }

Let $A=(a_{nk})$ $n,k=1,2,\ldots$ be an infinite matrix with complex entires and $c_A=\{x:Ax\in c\}$. Then $c_A$ is an FK-space with seminorms $\rho_0(x)=\sup_n \left| \sum_{k=1}^{\infty}a_{nk}x_k \right|$ $(n=1,2,\ldots)$, $\rho_n(x)=|x_n|$, $(n=1,2,\ldots)$; and $h_n(x)=\sup_m\left| \sum_{k=1}^{m}a_{nk}x_k \right|$ $(n=1,2,\ldots)$. Also, every $f\in c'_A$ if and only if 
\begin{equation*}
f(x)=\sum_{k=1}^{\infty}\beta_kx_k+\sum_{n=1}^{\infty}t_n\sum_{k=1}^{\infty}a_{nk}x_k+\mu\lim{_A}x \ ,
\end{equation*}
where $t\in l$, $\mu\in\mathbb{C}$, $(\beta_k)\in c^{\beta}_A$, the $\beta$-dual of $c_A$ \cite{wilansky}. The representation is not unique; we say that $A$ is $\mu$-unique if all representations for some $f$ have the same $\mu$. If $A$ be $\mu$-unique, $c_A\subset c_D$, $D$ is conull with respect to $A$ if and only if $\mu_A(\lim_D)=0$ in \cite{wilansky4}.

Let $X$ and $Y$ be FK-spaces, $X$ with paranorm $\rho$ and $Y$ with paranorm $s$. It is shown that $Z=X+Y$ with the unrestricted inductive limit topology is an FK-space as in Theorem 4.5.1 of \cite{wilansky}. The paranorm $\tau$ of $Z$ is given by 
\begin{eqnarray}
\tau(z)=\inf_{\substack{x+y=z\\
x\in X, y\in Y}  }(\rho(x)+s(y)).  \nonumber
\end{eqnarray}
Let $\{X^n\}_{n=1}^{\infty}$ be a sequence of FK-spaces. $\rho_n$ the paranorm of $X^n$ and $\{s_{nk}\}_{k=1}^{\infty}$ be the seminorms of $X^n$. Let $Y=\underset{n}{\bigcap}X^n$. It is well known that $Y$ is an FK-space with paranorm $s=\sum_{n=1}^{\infty}\frac{\rho_n}{2^n(1+\rho_n)}$ and seminorms $\{s_{nk}\}_{n,k=1}^{\infty}$.

We now investigate some important subspaces of a locally convex FK-space $X$ containing $\phi$ which are analogous to these give in \cite{devos}. To prove the theorems of this section we use the same technique by DeVos in \cite{devos}.

\begin{theorem}
	\label{E(X)}
	Let $X$, $Y$ be FK-spaces and $Z=X+Y$. Then $E(X)+E(Y)\subseteq E(Z)$ for $E=D_{p}^{q}S$, $D_{p}^{q}W$, $D_{p}^{q}F$ or $D_{p}^{q}B$.
\end{theorem}

\begin{proof}
Let $E=D_{p}^{q}S$. We take $x\in D_{p}^{q}S(X)$ and $y\in D_{p}^{q}S(Y)$. Then 
\begin{equation*}
\rho\left( \frac{1}{q(n)-p(n)}\sum_{k=p(n)+1}^{q(n)}x^{(k)}-x \right)\to 0 ~~ \text{  and  } ~~ s\left( \frac{1}{q(n)-p(n)}\sum_{k=p(n)+1}^{q(n)}y^{(k)}-y  \right)\to 0
\end{equation*}	
as $n\to\infty$. Hence 
\begin{eqnarray}
&&r\left( \frac{1}{q(n)-p(n)}\sum_{k=p(n)+1}^{q(n)}(x+y)^{(k)}-(x+y) \right) \nonumber \\
&&\leq \rho\left( \frac{1}{q(n)-p(n)}\sum_{k=p(n)+1}^{q(n)}x^{(k)}-x \right)+s\left( \frac{1}{q(n)-p(n)}\sum_{k=p(n)+1}^{q(n)}y^{(k)}-y  \right)\nonumber
\end{eqnarray}
which implies that $x+y\in D_{p}^{q}S(Z) $.

Let $E=D_{p}^{q}W$. We take $x\in D_{p}^{q}W(X)$, $y\in D_{p}^{q}W(Y)$ and $f\in Z'$. Then $f|X\in X'$ and $f|Y\in Y'$.
\begin{eqnarray}
&&f(x+y)=f(x)+f(y)  \nonumber \\
&&=\lim_n \frac{1}{q(n)-p(n)}\sum_{k=p(n)+1}^{q(n)} \sum_{j=1}^{k}f(\delta^j)x_j+\lim_n \frac{1}{q(n)-p(n)}\sum_{k=p(n)+1}^{q(n)} \sum_{j=1}^{k}f(\delta^j)y_j \nonumber\\
&&=\lim_n \frac{1}{q(n)-p(n)}\sum_{k=p(n)+1}^{q(n)} \sum_{j=1}^{k}f(\delta^j)(x_j+y_j) \ . \nonumber  
\end{eqnarray}
The proofs for $E=D_{p}^{q}F$ or $D_{p}^{q}B$ are similar, so the details are omitted. 
\end{proof}	

\begin{theorem}
	\label{bigcap}
	Let $\{X^n\}_{n=1}^{\infty}$ be a sequence of FK-spaces and $Y=\underset{n}{\bigcap}X^n$. Then $E(Y)=\underset{n}{\bigcap} E(X^n)$ for $E=D_{p}^{q}S $, $D_{p}^{q}W$, $D_{p}^{q}F$ or $D_{p}^{q}B$.
\end{theorem}

\begin{proof}
	By Theorem \ref{monotone}, for each $n$, $E(Y)\subseteq E(X^n)$, hence $E(Y)\subseteq \underset{n}{\bigcap} E(X^n)$ for $E=D_{p}^{q}S $, $D_{p}^{q}W$, $D_{p}^{q}F$ or $D_{p}^{q}B$.
	
	Let $z\in \underset{n}{\bigcap} D_{p}^{q}S(X^n)$. Then $s_{nk}\left( \frac{1}{q(r)-p(r)}\sum_{n=p(r)+1}^{q(r)}z^{(n)}-z \right)\to 0$, $r\to\infty$, for each fixed $n$ and $k$, but these are the seminorms for $Y$. Hence 
	\begin{equation*}
	\lim_{r\to\infty}\frac{1}{q(r)-p(r)}\sum_{n=p(r)+1}^{q(r)}z^{(n)}=z  \text{  in  } Y
	\end{equation*}
	which implies that $z\in D_{p}^{q}S(Y) $. 
	
	Let $z\in \underset{n}{\bigcap} D_{p}^{q}W(X^n)$ and $f\in Y'$. Then we have $f=\sum_{j=1}^{h}f_j$, where $f_j\in (X^j)'$ (see \cite{wilansky4}; Sections 4.4 and 11.3). Since $f_j\left( \frac{1}{q(r)-p(r)}\sum_{n=p(r)+1}^{q(r)}z^{(n)} \right) \to f_j(z)$ for $j=1,2,\ldots,h$. Therefore,
	\begin{equation*}
	f\left( \frac{1}{q(r)-p(r)}\sum_{n=p(r)+1}^{q(r)}z^{(n)} \right) \to f(z)\ .
	\end{equation*}
	Hence $z\in D_{p}^{q}W(Y) $. 
	
	The proof for $E=D_{p}^{q}F$ is similar to previous paragraph, so we omit the details. 
	
	Let $z\in \underset{n}{\bigcap} D_{p}^{q}B(X^n)$. Then for any fixed $l$ and $k$ , $s_{lk}\left( \frac{1}{q(r)-p(r)}\sum_{n=p(r)+1}^{q(r)}z^{(n)} \right)\leq H_{lk}$ for all $r$. Hence $z\in D_{p}^{q}B(Y) $. 
\end{proof}

In \cite{sezgek}, let $X$ be an FK-space containing $\phi_1$ and 
\begin{eqnarray}
\label{gamma}
\zeta^n:=e-\frac{1}{q(n)-p(n)} \sum_{k=p(n)+1}^{q(n)}e^{(k)}   
\end{eqnarray}
\begin{eqnarray}
= \left( 0,0,\ldots,0,\frac{1}{q(n)-p(n)}, \frac{2}{q(n)-p(n)}, \frac{3}{q(n)-p(n)}, \ldots , \frac{q(n)-p(n)-1}{q(n)-p(n)},1,1 , \ldots \right) . \nonumber 
\end{eqnarray}	

If $\zeta^n\to 0$ in $X$, then $X$ is called strongly deferred Ces\`{a}ro conull, where $e^{(k)}:=\sum_{j=1}^{k}\delta^j$. If the convergence holds in the weak topology in \eqref{gamma} then $X$ is called deferred Ces\`{a}ro conull. Hence $X$ is deferred Ces\`{a}ro conull iff
\begin{equation*}
f(e) = \lim_n \frac{1}{q(n)-p(n)}\sum_{k=p(n)+1}^{q(n)}\sum_{j=1}^{k}f(\delta^j)   ~~,  ~~ \  \ \ \forall f\in X'.
\end{equation*}

Now, we define deferred Ces\`{a}ro conullity of one FK-space with respect to another. 

\begin{definition}
	\label{conull}
	Let $X$ be an FK-space with $D_{p}^{q}W(X)\neq D_{p}^{q}B(X)$ and $Y$ be an FK-space , $X\subseteq Y$. $Y$ is deferred Ces\`{a}ro conull with respect to $X$ iff $D_{p}^{q}B(X)\subseteq D_{p}^{q}W(Y)$.
\end{definition}

\begin{theorem}
	Let $X$, $Y$, $Z$ be FK-spaces with $X\subseteq Y \subseteq Z$. Then
	
	i) If  $\ Y$ is deferred Ces\`{a}ro conull with respect to $X$ then $Z$ is deferred Ces\`{a}ro conull with respect to $X$, 
	
	ii) If $Z$ is deferred Ces\`{a}ro conull with respect to $X$ and $Y$ is closed in $Z$ then $Y$ is deferred Ces\`{a}ro conull with respect to $X$.   
\end{theorem}
The proof of Theorem is clear by Definition \ref{conull} and Theorem \ref{monotone}. 
\begin{theorem}
	Let $\{Y^n\}_{n=1}^{\infty}$ be FK-spaces such that each $Y^n$ is deferred Ces\`{a}ro conull with respect to $X$. Then $\underset{n}{\bigcap}Y^n$ is deferred Ces\`{a}ro conull with respect to $X$. 
\end{theorem}
The proof of Theorem is obtained by Definition \ref{conull} and Theorem \ref{bigcap}. 

Let $E(c_A)=E(A)$ for $A$ a matrix and $E= D_{p}^{q}W$ or $D_{p}^{q}B$ and $\mu_A (\lim_D)=\mu_A(D)$. For many cases the following theorem gives an equivalence between Wilansky's and our extensions of deferred Ces\`{a}ro conullty. 

\begin{theorem}
	Let $A$ and $D$ be matrices with $D_{p}^{q}W(A)\neq D_{p}^{q}B(A)$ and $c_A\subset c_D$. $\mu_A(D)=0$ if and only if $c_D$ is deferred Ces\`{a}ro conull with respect to $c_A$. 
\end{theorem}

\begin{proof}
Let $c_D$ is deferred Ces\`{a}ro conull with respect to $c_A$. For $x\in D_{p}^{q}B(A)$, $\lim_Dx=\mu_A(D)\lim_Ax+\beta x$. 

Now let $z\in D_{p}^{q}B(A) \setminus D_{p}^{q}W(A) $. Firstly, by Theorem 4.2 of \cite{sezgek} we have 
\begin{equation*}
\lim{_A}  \left( z-\frac{1}{q(n)-p(n)}\sum_{k=p(n)+1}^{q(n)}z^{(k)} \right)\nrightarrow 0 \ .
\end{equation*}	
Also, since $D_{p}^{q}B(A)\subset c_A$ and $\gamma\in c_A^{\beta}$ we obtained 
\begin{eqnarray}
\gamma \left(  z-\frac{1}{q(n)-p(n)}\sum_{k=p(n)+1}^{q(n)}z^{(k)} \right)&=&\gamma \left(  z-\frac{1}{q(n)-p(n)}\sum_{k=p(n)+1}^{q(n)}\sum_{j=1}^{k}z_j\delta^j \right) \nonumber \\
&=& \frac{1}{q(n)-p(n)}\sum_{k=p(n)+1}^{q(n)}\sum_{j=k+1}^{\infty}z_j\delta^j \ . \nonumber
\end{eqnarray}
By hypothesis, for each $f\in (c_D)'$, 
\begin{equation*}
 f\left(  z-\frac{1}{q(n)-p(n)}\sum_{k=p(n)+1}^{q(n)}z^{(k)} \right)\to 0 \ .
\end{equation*}
In particular, we take $f=\lim_D\in (c_D)'$. Thus $\lim_D\left(  z-\frac{1}{q(n)-p(n)}\sum_{k=p(n)+1}^{q(n)}z^{(k)} \right)\to 0 $.

Conversely let $f\in (c_D)' $. By Theorem 5.2 of \cite{wilansky4} we have $\mu_A(f)=\mu_D(f).\mu_A(D)=0$. Hence $f(x)=t(Ax)+\beta x$ for $x\in c_A$. 

Now we are able to write 
\begin{equation*}
f(x)=\mu_A(f)\lim{_A}x+\beta x
\end{equation*}
with $\gamma=tA+\beta $ for $x\in D_{p}^{q}B$ by Corollary 12.5.9 of \cite{wilansky}. Therefore, we get $f(x)=\gamma x$ for $x\in D_{p}^{q}B(A)$ which implies that $x\in D_{p}^{q}W(D)$. So $c_D$ is deferred Ces\`{a}ro conull with respect to $c_A$. 
\end{proof}

We establish some relations among the subspaces $D_{p}^{q}S$, $D_{p}^{q}W$, $D_{p}^{q}F$, $D_{p}^{q}F^+$.

\begin{remark}
	\label{weakly}
	Let $X$ be an FK-space such that weakly convergent sequences are convergent in the FK-topology, $A$ be a matrix such that $X_A\supset \phi$. The subspaces $D_{p}^{q}S$, $D_{p}^{q}W$ and $D_{p}^{q}F$ are calculated in $X_A$.  
\end{remark}

\begin{lemma}
	If $X$ is as in Remark \ref{weakly}, then for $X$ itself, we have $D_{p}^{q}S=D_{p}^{q}W=D_{p}^{q}F=D_{p}^{q}F^+$.
\end{lemma}

\begin{proof}
	The inclusions $D_{p}^{q}S\subset D_{p}^{q}W\subset D_{p}^{q}F\subset D_{p}^{q}F^+$ are trivial by definitions. Conversely, if $x\in D_{p}^{q}F^+$, then $\left\{ \frac{1}{q(n)-p(n)}\sum_{k=p(n)+1}^{q(n)}x^{(k)} \right\}$ is weakly Cauchy, hence by (12.0.1) of \cite{wilansky} is Cauchy in the FK-topology of $X$, so convergent, say \begin{equation*}
	\left\{ \frac{1}{q(n)-p(n)}\sum_{k=p(n)+1}^{q(n)}x^{(k)} \right\}\to y \ .
	\end{equation*}
	Since $x^{(k)}\to x$ in $w$, we have \begin{equation*}
	\left\{ \frac{1}{q(n)-p(n)}\sum_{k=p(n)+1}^{q(n)}x^{(k)} \right\}\to x \ ~~ \text{  in } w \ .
	\end{equation*}
	 By the contiunity of $i:X\to w$, $y=x$, and $x\in D_{p}^{q}S$.
\end{proof}	

Now we note that if $X$ is an FK-space containing $\phi_1$,  then \begin{equation}
\label{fdual}
D_{p}^{q}F^+=X^{fd} \ .
\end{equation}
To see this, it is enough to take $\sigma_{p}^{q} [s]$ instead of $cs$ in Theorem 10.4.2 of \cite{wilansky}. If $X$ is also $\sigma_{p}^{q} [K]$, $X^{dd}=X$ since, by Theorem 1.9 of \cite{goes3}, $X^{dd}=X^{fd}$ and $D_{p}^{q}F^+=X^{fd}$ by \eqref{fdual}. We have $X^{dd}=D_{p}^{q}F=D_{p}^{q}F^+\subset X$, hence, the result follows. 

\begin{theorem}
	\label{equal}
	With $X$, $A$ as in Remark \ref{weakly}, for the $X_A$, we have $D_{p}^{q}S=D_{p}^{q}W=D_{p}^{q}F=D_{p}^{q}F^+$.
\end{theorem}
\begin{proof}
If $x\in D_{p}^{q}F^+$, then $\left\{ \frac{1}{q(n)-p(n)}\sum_{k=p(n)+1}^{q(n)}x^{(k)} \right\}$ is weakly Cauchy, hence by (12.0.1) of \cite{wilansky} is Cauchy in the FK-topology of $X$, so convergent. Since by Corollary 4.2.4 of \cite{wilansky} the matrix mapping $A:X_A\to X$ is continuous, $\left\{ \frac{1}{q(n)-p(n)}\sum_{k=p(n)+1}^{q(n)}Ax^{(k)} \right\}$ is convergent in $X$, say \begin{equation*}
\left\{ \frac{1}{q(n)-p(n)}\sum_{k=p(n)+1}^{q(n)}Ax^{(k)} \right\}\to y \ .
\end{equation*}
On the other hand, by Theorem 4.3.8 of \cite{wilansky} $(w_A,\rho\cup h) $ is an AK-space. Hence it is also $\sigma_{p}^{q} [K]$-space. Hence \begin{equation*}
\left\{ \frac{1}{q(n)-p(n)}\sum_{k=p(n)+1}^{q(n)}x^{(k)} \right\}\to x \ .
\end{equation*}
The matrix mapping $A:w_A\to w$ is continuous, and therefore \begin{equation*}
\left\{ \frac{1}{q(n)-p(n)}\sum_{k=p(n)+1}^{q(n)}Ax^{(k)} \right\}\to Ax \ \text{  in  } w.
\end{equation*} 
Since $X\subset w$ and $X$ is complete, $Ax=y$. We have 
\begin{equation*}
\left\{ \frac{1}{q(n)-p(n)}\sum_{k=p(n)+1}^{q(n)}Ax^{(k)} \right\}\to Ax \ \text{  in  } X.
\end{equation*} That is,
\begin{equation*}
r\left(  Ax- \frac{1}{q(n)-p(n)}\sum_{k=p(n)+1}^{q(n)}Ax^{(k)}\right)=(r\circ A)\left( x- \frac{1}{q(n)-p(n)}\sum_{k=p(n)+1}^{q(n)}x^{(k)} \right)\to 0
\end{equation*}
as $n\to\infty$, where $r$ is a typcal seminorm of $X$. Hence $x\in D_{p}^{q}S $, which proves Theorem \ref{equal}.
\end{proof}
\ 
 
\subsection{Replaceability, Deferred Ces\`{a}ro Conullity of $l_A$}
~~

Recall that given a matrix $A$ with $l_A\supset \phi$ is called $l$-replaceable if there is a matrix $D=(d_{nk})$ with $l_D=l_A$, and $\sum_nd_{nk}=1$ $k\in\mathbb{N}$ \cite{macphail}. It is easy to see that $A$ is replaceable if and only if there exists $f\in l'_A$ with $f(\delta^k)=1 \ (k\in \mathbb{N})$, namely, $f=\sum_D$.
\begin{theorem}
	Suppose that $D_{p}^{q}F=l_A$. Then $A$ is $l$-replaceable if and only if $l_A\subset \sigma_{p}^{q} [s] $. 
\end{theorem}
\begin{proof}
Assume that $A$ is $l$-replaceable. Then it follows from \cite{macphail} that $A$ is $l$-replaceable if and only if there is $f\in l'_A$ such that $f(\delta^j)=1$ for all $j\in \mathbb{N}$. Since $D_{p}^{q}F=l_A$, 
\begin{eqnarray}
\lim_n\frac{1}{q(n)-p(n)}\sum_{k=p(n)+1}^{q(n)}\sum_{j=1}^{k}x_jf(\delta^j)=\lim_n \frac{1}{q(n)-p(n)}\sum_{k=p(n)+1}^{q(n)}\sum_{j=1}^{k}x_j \nonumber 
\end{eqnarray}
exists for all $x\in l_A$, hence $x\in \sigma_{p}^{q} [s]$.

Conversely, if $l_A\subset \sigma_{p}^{q} [s]$ then
\begin{equation*}
f(x)= \frac{1}{q(n)-p(n)}\sum_{k=p(n)+1}^{q(n)}\sum_{j=1}^{k}x_j
\end{equation*}
 defines an element $f$ of $l'_A$, and we get $f(\delta^{\nu})=1$, $(\nu=1,2,\ldots)$.
\end{proof}	

We now establish a relation between deferred Ces\`{a}ro conullity and replaceability.
\begin{theorem}
	If $l_A$ is deferred Ces\`{a}ro conull space, then $A$ is not $l$-replaceable.
\end{theorem}
\begin{proof}
Suppose that $A$ is $l$-replaceable. Then it follows from \cite{macphail} that $A$ is $l$-replaceable if and only if there is $f\in l'_A$ such that $f(\delta^{j})=1$ for all $j\in\mathbb{N}$. Hence 
\begin{eqnarray}
&&f(e)-\frac{1}{q(n)-p(n)}\sum_{k=p(n)+1}^{q(n)}\sum_{j=1}^{k} f(\delta^j) = \left( f(e)- \frac{1}{q(n)-p(n)}\sum_{k=p(n)+1}^{q(n)}\sum_{j=1}^{k}(1)   \right) \nonumber \\
&&=  \left( f(e)- \frac{1}{q(n)-p(n)}\sum_{k=p(n)+1}^{q(n)}(k)  \right) \nonumber \\
&&=\left( f(e)-\frac{1}{q(n)-p(n)} \frac{(q(n)-p(n))(q(n)+p(n)+1)}{2}  \right) \nonumber \\
&&=\left( f(e)- \frac{q(n)+p(n)+1}{2}  \right)\nrightarrow 0 \ \text{  as  } n\to\infty \ , \nonumber 
\end{eqnarray}
so, $l_A$ is not deferred Ces\`{a}ro conull.
\end{proof}

\end{document}